\documentclass[a4paper,11pt]{article}

\usepackage[utf8]{inputenc}    
\usepackage{times}             
\usepackage[english]{babel}     
\usepackage[T1]{fontenc}
\usepackage{amsmath, amsthm, amssymb} 
\usepackage{dsfont}  
\usepackage{graphicx} 
\usepackage{textcomp}  
\usepackage{enumerate}     
\usepackage{authblk}                        
\newtheorem{lemma}{Lemma}
\newtheorem{cons}{Corollary}
\newtheorem{theo}{Theorem}
\newtheorem*{cons*}{Corollary}
\newtheorem*{theo*}{Theorem}

\theoremstyle{definition}
\newtheorem{definition}{Definition}
\newtheorem{process}{Process}

\theoremstyle{remark}

\newcommand{\R}{\mathbb{R}}

\newcommand{\C}{\mathcal{C}}
\renewcommand{\Pr}{\mathbb{P}}
\newcommand{\Esp}{\mathbb{E}}
\renewcommand{\P}{\mathbf{P}}
\newcommand{\pper}{p^{\mathrm{per}}}
\newcommand{\ptable}{p^{\mathrm{table}}}
\newcommand{\Cp}{\mathcal{C}^{\mathrm{per}}}
\newcommand{\Ca}{\mathcal{C}}

\title{A probabilistic Hadwiger-Nelson problem}

\author[1]{Thomas Bourgeat}
\author[1]{Marc Heinrich}
\author[1]{Paul Melotti}
\affil[1]{Computer science department,\authorcr
\'Ecole Normale Sup\'erieure, Paris, France\authorcr
		  \texttt{\{first name\}.\{name\}@ens.fr}
}

\author[2]{Jean-Marc Robert}
\affil[2]{Software engineering and IT department,\authorcr
\'Ecole de Technologie Sup\'erieure, Montr\'eal, Canada\authorcr
		  \texttt{jean-marc.robert@etsmtl.ca}
}

\begin{document}
\maketitle

\begin{abstract} If you color a table using $k$ colors, and throw a needle
randomly on it, for some proper definition, you get a certain probability that the endpoints will fall on 
different colors. How can one make this probability maximal?
This problem is related to finite graphs having unit-length edges, and 
some bounds on the optimal probability are deduced.\end{abstract}

\section{Introduction}

A well-known problem in geometric graph theory is the following: how many colors
are needed to color the Euclidian plane so that no two points at unit distance have the 
same color? This is known as the Hadwiger-Nelson problem. It is often stated in 
the language of graph theory: let $E^2$ be the graph whose set of vertices is $\R^2$,
and so that two vertices are joined by an edge \textit{iff} their euclidian distance
is 1. Then what we are looking for is the \textit{chromatic number} of $E^2$.

It seems that the problem was first stated by Nelson in 1950, and first published by 
Gardner in \cite{gardner}.
For a nice review about this problem and many others, see Soifer's book
\cite{soifer}. This question is still very mysterious to that day, and all that 
is known for sure is that the chromatic number of the plane is not greater than
7 and not smaller than 4. The upper bound is derived from a tesselation of the 
plane with regular hexagons proposed by Isbell \cite{soifer}. On the other hand,
the lower bound is obtained from the \textit{Moser spindle}.\cite{moser}, a graph with edges
of length 1 which cannot be colored with three colors (see Section \ref{3col} for
a definition of this graph).

The questions presented here may be seen as a probabilistic version of this 
problem, since we do not try to make \textit{every} unit-length segment 
bi-chromatic but a large fraction of them.

\medskip

Let's state our results in a rough, imprecise way: a $k$-\textit{coloring} is a Borelian 
function $c$ from, say, a square, to the finite set $\{1, \dots , k \}$, and let $p(c)$ be the probability that when throwing a unit-length
segment (or \textit{needle}) on $c$, both endpoints of the needle have the same color.
Our aim is to find a coloring $c$ such that $p(c)$ is as small as possible.

For a finite graph $g$, let $m_k(g)$ be the smallest fraction of monochromatic
edges in a $k$-coloring of $g$. Then, let $\mathcal{G}$ be the set of 
\textit{unit distance graphs} -- that is, graphs that can be embedded in the
Euclidian plane with all edges of unit-length, or equivalently, such that there is
a graph morphism from $g$ to $E^2$ -- and let $\C_k$ be the set of $k$-colorings. We
have:
\begin{theo*}
$$ \sup_{g \in \mathcal{G}} \ m_k(g) \leq \inf_{c \in \C_k} \ p(c) \leq \frac{1}{k}. $$
\end{theo*}
We do not know whether the left hand-side of the inequality is actually an equality in
the general case or not, but we prove that it is so for $k=2$.

We also prove the following consequence, regarding the Hadwiger-Nelson problem:
\begin{cons} \label{con}
If \ $\inf_{c \in \Ca_k} \ p(c) = 0$ \ then the plane can be $k$-colored.
\end{cons}

The paper is organized as follows: in Section~\ref{defsection} we clarify the
definitions and tools we will use. Section~\ref{equiv} is devoted to the proof 
of the theorem in the case of periodic colorings of the plane; a generalization
to what we call ``asymptotic colorings'' is rejected to Section~\ref{ext}.
In Section~\ref{appli} we give some applications, studying the particular cases $k=2,3$,
and we prove the corollary.
We study in Section~\ref{ext} the case of a finite table instead of the whole plane, 
and give some remarks on possible 
generalizations to higher dimensions.

\section{Definitions} \label{defsection}
We will be interested in \textit{periodic} and \textit{asymptotic} colorings,
the latter being a generalization of the former.
\begin{definition}
A $k$-\textit{coloring} of the plane is a measurable function $c$ from the Euclidean 
plane~$\R ^2$ to the finite set $\{1, \dots , k \}$.

A \textit{coloring} of the plane is simply is a $k$-coloring for some $k \geq 2$.
\end{definition}

\begin{definition}
\
\begin{enumerate}[i)] \label{pper}
\item A coloring $c$ is said to be \textit{periodic} if there are two free
vectors $u,v\in \R^2$ so that for any $x$ in $\R ^2$, $c(x+u)=c(x+v)=c(x)$. In such a
case let $\P_c$ be the parallelogram formed on $u$ and $v$, \textit{i.e.}:
\[\P_c = \{tu+sv, \ 0\leq t,s< 1 \} \subset \R ^2. \]
\item For a periodic coloring $c$, consider the following random variables:
$A$ is a random point chosen uniformly in $\P_c$, 
$\theta$ is a random angle taken uniformly in $[0;2 \pi[$, independent of $A$, 
and $B=A + e^{i \theta}$.
Then let $\pper(c)$ denote the probability
$$\pper(c) = \Pr(c(A)=c(B)).$$
\end{enumerate}
\end{definition}

\begin{definition}
\label{defas}For a general coloring $c$, and $R>1$, consider the 
following random variables:
$A_R$ is a random point chosen uniformly in $[-R;R]^2$, 
$\theta$ is a random angle taken uniformly in $[0;2 \pi[$, independent of $A_R$, 
and $B_R=A_R + e^{i \theta}$.
Then let $\ptable_R(c)$ denote the conditional probability
$$\ptable_R(c) = \Pr \left( c(A_R)=c(B_R) \ \mid \ B_R \in [-R;R]^2 \right),$$
which is well defined because $R>1$.

The coloring $c$ is said to be \textit{asymptotic} if $\ptable_R(c)$
converges to a limit as $R$ tends to infinity. This limit is called $p(c)$.
\end{definition}

It should be clear that a periodic coloring $c$ is also an asymptotic coloring,
and $\pper(c)=p(c)$ in that case. However, $p(c)$ is not
in general a probability and should not be seen as such; we may call it an 
asymptotic probability.
\smallskip

Let's turn to the definitions regarding finite graphs, that describe how a graph
can be ``well $k$-colored'' even when its chromatic number is greater
than $k$.
\begin{definition}
\ 
\begin{enumerate}[i)]
\item For a finite graph $G=(V,E)$ and a $k$-coloring $c$ of the vertices of $G$, 
let $M_G(c)$ be the number of monochromatic edges in $G$ when colored with $c$, \textit{i.e.}
$$M_G(c) = \sum_{(i, j) \in E} \mathds{1}_{c(i)=c(j)}.$$
where the function $\mathds{1}_{c(i)=c(j)}$ returns 1 if $c(i)=c(j)$ and 0 otherwise. Here the sum is taken on the edges of $G$, meaning that every edge appears exactly one time in the sum.

Then let $m_k(G)$ be
\[m_k(G) = \min_{c \ \mathrm{is \ a} \ k-\mathrm{coloring \ of} \ G} \frac{M_G(c)}{|E|}.\]
\item A finite graph is said to be a \emph{unit-distance graph} if it has 
an embedding in $\R^2$ in which all edges have unit-length, that we will call an \textit{embedding in} $E^2$.
\end{enumerate}
\end{definition}
Now the main result of this paper can be properly stated:
\begin{theo} \label{ineg}
Let $\mathcal{G}$ be the set of unit distance graphs and $\Cp_k$ the set 
of periodic $k$-colorings of the plane, then
$$ \sup_{g \in \mathcal{G}} \ m_k(g) \leq \inf_{c \in \Cp_k} \ \pper(c) \leq \frac{1}{k}. $$
\end{theo}
We will then deduce an analogous theorem for asymptotic 
colorings:
\begin{theo} \label{ineg2}
Let $\mathcal{G}$ be the set of unit distance graphs and $\C_k$ the set 
of asymptotic $k$-colorings of the plane, then
$$ \sup_{g \in \mathcal{G}} \ m_k(g) \leq \inf_{c \in \C_k} \ p(c) \leq \frac{1}{k}. $$
\end{theo}

\section{Process equivalence}
\label{equiv}

Let $c$ be a periodic coloring of the plane, and consider the following needle-throwing ``processes''. The first is only a rephrasement of Definition~\ref{pper}, given here for clarity purposes.
\begin{process} \label{premier}
Consider the following random variables:
\begin{itemize}
  \item a random point $A$ chosen uniformly 
  in $\P_c$ ;
  \item an independent random angle $\theta$ chosen uniformly in $[0;2 \pi[$, and define $B$ as $A + e^{i \theta}$.
\end{itemize}
\end{process}

Note that $B$ may fall outside of $\P_c$. In that case, $c(B)$ is 
defined using the periodicity of the coloring.
Our goal is to evaluate and minimize the probability that both ends have the 
same color.

Now consider the second process that will be very useful in our 
proofs. The idea is to throw a unit distance graph on the 
plane, and then to choose a needle on that graph:
\begin{process}
Given a unit distance graph $G = (V,E)$, embedded in $E^2$. Label its edges with numbers from 
$1$ to $m$. The complex coordinates of the vertices of the $j^{th}$ edges are 
named $z_j $ and $z_j + e^{i.\theta_j}$. Then let: 
\begin{itemize}
\item $A_0$ be a random point taken uniformly in $\P_c$ ;
\item $\theta$ be an independent angle taken uniformly in $[0;2\pi[$ (to rotate the graph) ;
\item $J$ be an integer in $[| 1;m|]$, independent from all of the above.
\end{itemize}
Then, rotate the graph $G$ by the angle $\theta$, translate it so that the origin 
of the plane falls on $A_0$, and take the needle $(A',B')$ corresponding to the $J^{th}$ 
edge of the obtained graph. In other words, the endpoints of the needle have the law:
$$A' = e^{i.\theta} z_J + A_0, \ B' = e^{i.\theta} (z_J + e^{i.\theta_J}) + A_0.$$
\end{process}

\begin{lemma}\label{huitre}
Let $c_1$ and $c_2$ denote 
two colors. Then, we 
have~:\\
 $$\mathbb{P}(c(A) = c_1, c(B) = c_2) = \mathbb{P}(c(A') = c_1, c(B') = c_2) $$ \\
 In particular, this probability does not depend on the graph 
 chosen.
\end{lemma}

\begin{proof}
We call $\mathcal{A}(\P_c)$ the area of $\P_c$.

We prove the lemma by the following series of equalities:
\begin{eqnarray*}
& & \mathbb{P}(c(A') = c_1, c(B') = c_2) \\
  &=& \frac{1}{m}\sum_{j=1}^{m} \mathbb{P}(c(A') = c_1, c(B') = c_2 | J=j)  \\
  &=& \frac{1}{m}\sum_{j=1}^{m}  \frac{1}{2\pi}\int_{\theta =0} ^{2\pi} \frac{1}{\mathcal{A}(\P_c)}\int_{A_0 \in \P_c} \mathbb{P}(c(A') = c_1, c(B') = c_2 | J=j, \theta, A_0) d\theta dx dy \\  
  &=& \frac{1}{m 2\pi \mathcal{A}(\P_c)}\sum_{j=1}^{m}\int_{\theta =0} ^{2\pi} \int_{A_0 \in \P_c} \mathds{1}(c(A_0 + z_i e^{i\theta}) = c_1, c(A_0 +z_j e^{i\theta} + e^{i (\theta + \theta_j)}) = c_2) d\theta dx dy \\  
    &=& \frac{1}{m 2\pi \mathcal{A}(\P_c)}\sum_{j=1}^{m} \int_{\theta =0} ^{2\pi} \int_{A_0 \in \P_c} \mathds{1}(c(A_0) = c_1, c(A_0 + e^{i (\theta + \theta_j)}) = c_2 ) d\theta dx dy \hspace{1 cm} (*)\\ 
    &=& \frac{1}{m 2\pi \mathcal{A}(\P_c)}\sum_{j=1}^{m} \int_{\theta =0} ^{2\pi} \int_{A_0 \in \P_c} \mathds{1}(c(A_0) = c_1, c(A_0 + e^{i \theta}) = c_2 ) d\theta dx dy \\ 
    &=& \frac{1}{2\pi \mathcal{A}(\P_c)}\int_{\theta =0} ^{2\pi} \int_{A_0 \in \P_c} \mathds{1}(c(A_0) = c_1, c(A_0 + e^{i \theta}) = c_2 ) d\theta dx dy \\ 
    &=& \mathbb{P}(c(A) = c_1, c(B) = c_2)
\end{eqnarray*}

Equality $(*)$ is justified by the periodicity of our coloring.
\end{proof}

\begin{proof}[Proof of Theorem~\ref{ineg}]
Consider a graph $g$ colored with $k$ colors. When one chooses a needle
randomly on this graph the probability that both ends have the same
colors is clearly greater than $m_k(g)$. Applying lemma \ref{huitre} with $g$, 
it follows from the description of the second process that the probability 
$\pper(c)$ for any coloring $c$ is greater than $m_k(g)$.

Now to get the right-hand side of the theorem, let $\P$ be a square of side $R$,
and cut $\P$ into $n^2$ squares of side $R/n$. Let $C$ be a \textit{random} 
coloring obtained by assigning to these $n^2$ smaller squares i.i.d. colors
taken uniformly in $\{1,\dots,k\}$, and by repeating $\P$ to get a periodic 
coloring. Then one easily checks that for $R$ large enough and $R/n$ small
enough, for any needle $(A,B)$, $A$ and $B$ are in distinct small squares so
$$\Pr(C(A)=C(B)) = \frac1k,$$
and an application of Fubini's Theorem shows that
$$\Esp[\pper(C)] = \frac1k,$$
so that there is a realization $c_0$ with $\pper(c_0) \leq \frac1k$.
\end{proof}

\section{Applications} \label{appli}
\subsection{Connections with the Hadwiger-Nelson problem} \label{hn}
The results shown in this Section will be derived using the axiom of choice. 
This is important to notice, since the answer to the Hadwiger-Nelson problem is 
suspected to depend on the set of axioms used; see \cite{choice}, \cite{choice2}. We will use the 
De Bruijn - Erd\H{o}s theorem, established in \cite{erdos}, whose proof uses 
the axiom of choice.
\begin{theo}[De Bruijn - Erd\H{o}s]
 A graph $G$ can be colored with $k$ colors iff all of its finite subgraphs 
 can be colored with $k$ colors.
\end{theo}
In other words, the chromatic number of a graph is the maximum chromatic number 
of its finite subgraphs.

\begin{proof}[Proof of corollary~\ref{con}]If there exist a unit distance graph (finite by definition) which cannot 
be colored with $k$ colors, it 
follows from our study that the probability $p(c)$ for $c$ a valid 
$k$-coloring is always greater than a certain positive constant. Taking the 
contrapositive of this statement and using the De Bruijn-Erd\H{o}s theorem yield 
the corollary.
\end{proof}

\subsection{2 colors} \label{2col}

With two colors, we consider an equilateral triangle as a unit distance graph. 
As there are always at least two vertices with the same color, it's clear that 
$m_2(g) = \frac{1}{3}$. Thanks to Theorem~\ref{ineg}, 
we know that $\pper(c) \geq \frac13$ for any periodic coloring $c$.

This bound is optimal. Indeed, consider the coloring in Figure~\ref{couleur}, 
constructed with parallel strips of width 
$l = \frac {\sqrt3}{2}$, with the upper side opened and the lower side closed. 

\begin{figure}[h]
\center
\includegraphics[scale=0.5]{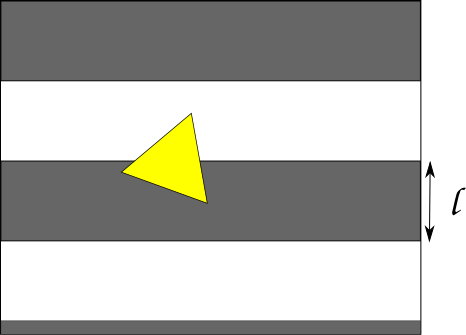}
\caption{\label{couleur} The parallel stripes $2$-coloring}
\end{figure}

For this coloring, the probability of getting the same color using the first 
process is $\frac13$. This can be shown by direct computation, or we can 
more simply remark that no unit-length equilateral triangle can have its 
three vertices with the same colors. Applying Lemma~\ref{huitre}, we conclude 
that this coloring achieves a probability of $\frac{1}{3}$ indeed.

\subsection{3 colors}\label{3col}
 With three colors, we use the Moser spindle of Figure~\ref{color} - a unit 
 distance graph with $m_3(g) = \frac{1}{11}$.
 We similarly get that for $k=3$, $p(c) \geq \frac{1}{11}$ for any valid coloring. 

\begin{figure}[h]
\center
\includegraphics[scale=0.4]{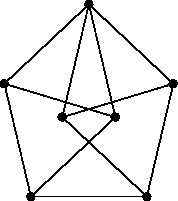}
\caption{\label{color} The Moser spindle}
\end{figure}

We do not know if the previous bound is optimal. We believe that Figure~\ref{trois} 
can give a rather good $3$-coloring of the plane, for some well chosen length 
of the edge of the hexagons. Rough simulations and optimization have shown that 
this coloring gives a $\pper(c)$ of about $0.13$ if the edge-length of the hexagons is about 
$0.61$, but $0.13$ is still greater than $\frac{1}{11} \simeq 0.091$.

This also tells us that a unit distance graph of chromatic number greater than 3
has at least $8$ edges, because otherwise Theorem~\ref{ineg} would imply that
all $3$-colorings of the plane satisfy $p(c) \geq \frac17$, and this is not the
case in our simulation. Thus, simulations can be used to provide lower bounds on the
number of edges of a non-$k$-colorable unit distance, for any $k$.

\begin{figure}[h]
\center
\includegraphics[scale=0.5]{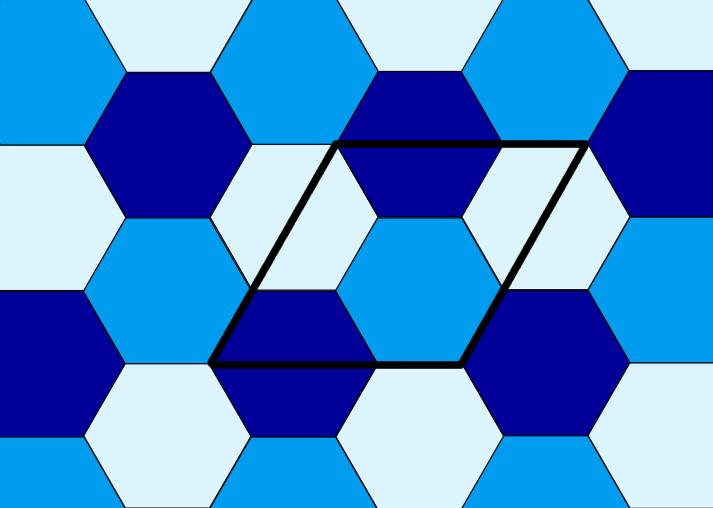}
\caption{\label{trois} A hexagonal $3$-coloring and its parallelogram of periodicity}
\end{figure}

\section{Extensions}
\label{ext}
\subsection{Finite table}
\label{fini}

We previously used periodic border conditions so that we didn't have to worry about the second 
endpoint $B$ falling outside of $\mathbf{P}$, which made things technically easier. We now 
try to require that $B$ fall in $\mathbf{P}$ to match a somewhat more practical 
view: we want to throw the needle on an ``actual'' bounded 
table. The first difficulty is to find a natural distribution for the needle.

\begin{process} \label{encore}
Denote by $\mathbf{P}$ an open parallelogram, representing the table. For the 
process to be well defined, we need to assume that $\mathbf{P}$ can contain at
least one needle.
\begin{itemize}
\item Let $(A,B)$ be a random needle on $\P$ as in Process~\ref{premier}. Let $(A'',B'')$ be random points in $\P$ following the law
of $(A,B)$ conditioned by the event $\{B \in \mathbf{P} \}$.
\end{itemize}
\end{process}

\begin{definition}
For $r>0$, let $K(\mathbf{P},r)$ the $r$-wide inner border of $\mathbf{P}$~:
\[K(\mathbf{P},r) = \{ z \in \mathbf{P} \mid d(z,\mathbf{P}^c) < r\} \]
where $d(z,\mathbf{P}^c)$ is the usual Euclidian distance from the point $z$ to the set $\mathbf{P}^c$.
\end{definition}

\begin{figure}[h]
\center
\includegraphics[scale=0.5]{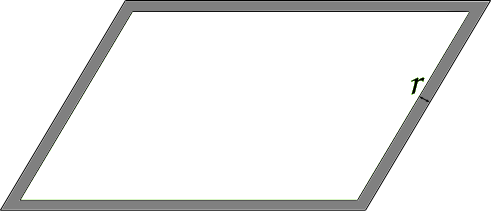}
\caption{\label{tablefinie} The set $K(\mathbf{P},r)$ in gray}
\end{figure}

\begin{lemma}\label{choco}
Let $c$ be a $k$-coloring of $\mathbf{P}$. Let
$\ptable(c) = \Pr(c(A'')=c(B''))$ be the probability given by Process~\ref{encore}. Then 
there is a positive universal constant $\kappa \leq 2$ so that 
$$ | \pper(\tilde{c}) - \ptable(c)| \leq \kappa \frac{\mathcal{A}(K(\mathbf{P},1))}{\mathcal{A}(\mathbf{P})}.$$
\end{lemma}

\begin{proof}
We will show that the inequality is true for $\kappa =2$, but this is probably 
not the best constant.

When $\mathcal{A}(K(\mathbf{P},1)) = \mathcal{A}(\mathbf{P})$ the result is clear,
so we may suppose that these areas differ. This implies that there are needles 
inside $\P$, so all conditional probabilities will be well defined.

Let $r = \frac{\mathcal{A}(K(\mathbf{P},1))}{\mathcal{A}(\mathbf{P})}$. When 
$r \geq \frac12$ the inequality is obvious, so we will now suppose that 
$r < \frac12$. Let $\tilde{c}$ be the periodic version of
$c$ given by repeating $\P$. Consider $(A,B)$ a needle whose law is given by Process~\ref{premier} on $\P$ and let $M$and $N$ be the following events:
$$ M = \{\tilde{c} (A) = \tilde{c} (B) \}, \ N = \{B\notin \P \}.$$
Then $P(N^c) \leq r$, indeed, $B$ can 
be outside of $\mathbf{P}$ only
when $A$ is in $K(\mathbf{P},1)$. Easy computation then shows that
\begin{eqnarray*}
\pper(\tilde{c}) - \ptable(c) & = & P(M) - P(M \mid N) \\
& = & \frac{P(N^c)}{P(N)} \left( P(M\mid N^c) - P(M)\right),
\end{eqnarray*}
and using the fact that $P(N^c) \leq r$ and 
$P(N) \geq 1-r \geq \frac12$ we get 
$$ |\pper(\tilde{c}) - \ptable(c)| \leq 2r.$$
\end{proof}

Since we already have bounds on $\pper$, we can therefore deduce bounds on $\ptable$.
For instance, with $2$ colors, we see that 
$\ptable(c) \geq \frac13 - \kappa \frac{\mathcal{A}(K(\mathbf{P},1))}{\mathcal{A}(\mathbf{P})}$. 
This bound becomes sharper as the table parallelogram $\mathbf{P}$ becomes bigger.
\begin{proof}[Proof of Theorem~\ref{ineg2}]
If $c$ is an asymptotic coloring, let $\tilde{c_R}$ be the periodic version of $c_{|[-R;R]^2}$ (that is, restrict $c$ to $[-R;R]^2$ and tile the plane by repeating this square). By Lemma~\ref{choco},
\begin{equation} \label{estim}
|\pper(\tilde{c_R}) - \ptable(c_{|[-R;R]^2}) | \leq \kappa \frac{\mathcal{A}(K([-R;R]^2,1))}{\mathcal{A}([-R;R]^2)} \leq \kappa \frac{4R}{R^2} = \frac{4\kappa}{R}.
\end{equation}
The probability $\ptable(c_{|[-R;R]^2})$ converges to $p(c)$ as $R\rightarrow +\infty$. The bounds on $\pper(\tilde{c_R})$ provided by Theorem~\ref{ineg} and the estimate~(\ref{estim}) yield Theorem~\ref{ineg2}.
\end{proof}

It may seem a bit odd to chose a square to define asymptotic colorings in definition~\ref{defas}. One could have, for instance, replaced $[-R,R]^2$ by the disk $D(O,R)$ in this definition. The advantage of the square is that it made the previous proof easier, given that periodic colorings were already studied. Getting similar results for a disk (or with any open set bounded by a smooth curve) shouldn't be harder, but it would require a result similar to Lemma~\ref{huitre} with an estimation similar to the one in Lemma~\ref{choco}.

However, it should be noted that the value of $p(c)$ defined by asymptotic colorings on different shapes may differ. For instance, consider the 2-coloring of Figure~\ref{contrex} where stripes are $\frac{\sqrt{3}}{2}$ apart. In the striped region the probability is approximately $\frac13$ (see Section \ref{2col}), and it is $1$ in the rest of the plane. So $p(c)$ is related to the fraction of a square (resp. disk) occupied by the striped region, and this fraction depends on the shape.

\begin{figure}[!ht]
\center
\includegraphics[scale=0.45]{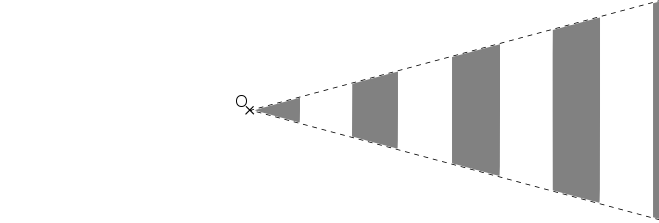}
\caption{\label{contrex} An asymptotic 2-coloring of the plane}
\end{figure}

The inequality of Theorem~\ref{ineg2} should nevertheless remain true for these different possible values of asymptotic $p(c)$'s.

\subsection{Higher dimensions}
\label{dim}
The previous considerations can be extended to dimension $d \geq 3$, however, 
it becomes more complicated to write down. We have to consider a basis of 
vector leaving our coloring unchanged; we still denote by $\mathbf{P}$ the 
parallelepiped induced by these vectors.

The first process for choosing a needle now consist in~: 
\begin{itemize}
	\item choosing one point uniformly in the parallelepiped $ \mathbf{P} $ as 
	one point of the needle~;
	\item choosing independently one point uniformly on the unit sphere 
	$S^{d-1}$, to give the relative position of the second point.
\end{itemize}

Sending a unit-length graph on the $d$-dimensional space becomes a bit more 
tricky. We have to choose one initial point in $ \mathbf{P} $, and an independent 
rotation of the graph; (by using the Haar measure of $SO(d)$). The second process now consists 
in: rotating the graph thanks to the matrix of $SO(d)$ chosen, translating it 
to the point of $\mathbf{P}$ chosen, and choosing one of the 
edges independently as your needle.

One can verify that the proof of Lemma~\ref{huitre} adapts to that definitions. 
Crucial 
points of the proof are the invariance of the Haar measure under matrix product, and the 
fact that the image measure of said Haar measure by the application 
$M\mapsto Mv$, where $v$ is a given unit vector, is the uniform Lebesgue 
measure of $S^{d-1}$.

Thus, to find lower bonds on our probability, the same techniques may be applied 
in any finite dimension.
%
Remark that in higher dimension, better minorants may be provided by 
Theorems~\ref{ineg}~and~\ref{ineg2}, since new graphs can 
become unit-length. For example, with $3$ colors in dimension $3$, the regular 
tetrahedron gives a lower bound 
of $\frac 1 6$, while we could only achieve a lower bound of $\frac 1 {11}$ in dimension $2$. More generally, consider the complete graph 
$K_{k+1}$ in dimension $k$ (also known as the regular $k+1$-simplex): it is a unit distance graph and is not $k$-colorable, so the 
probability of getting a monochromatic needle is at least $\frac{1}{\binom{k+1}{2}}$ 
for $k$ colors in dimension $k$.

\subsection*{Acknowledgements}
We are very grateful to David Naccache for his constant help in the writing of
this paper, and to Eric Brier for many ideas and insight on the subject of 
unit distance graphs.

\bibliographystyle{abbrv}
\bibliography{Tiling}

\end{document}